\documentclass{amsart}
\usepackage{amsmath,amsfonts,amsthm,amssymb}
\newtheorem{thm}{Theorem}

\newtheorem{fact}{Fact}
\title{Computing a Generating Set of Arithmetic Kleinian Groups}
\author{Gregory Muller}
\begin{document}

\maketitle

\begin{abstract}
The goal of this paper is to demonstrate the use of techniques from
hyperbolic geometry to compute generating sets of certain subgroups
of $SL^+(2,\mathbb{C})$; specifically, $SO^+(Q,\mathbb{Z})$ for $Q$
some integral quadratic form of signature $(3,1)$ that does not
represent $0$.
\end{abstract}

\section*{Introduction}

Hyperbolic space $\mathbb{H}^3$ can be thought of as the subset of
$\mathbb{R}^{4}$ that satisfies the equation $Q_1(\mathbf{x}) =
x_1^2+x_2^2+x_3^2-x_4^2=-1$ and the condition $x_4>0$, together with
the Riemannian metric induced by $Q_1$.  The group of isometries of
$\mathbb{H}^3$ is the space of $4$ by $4$ real matrices $A$ such
that $A^tQ_1A=Q_1$ and $e_4\cdot Ae_4>0$, where $e_4$ is the unit
vector in the $x_4$ direction.\footnote{The expression $A^tQ_1A=Q_1$
is using the fact that any quadratic form $Q$ can be written as a
symmetric matrix $\tilde{Q}$ such that
$Q(\mathbf{x})=\mathbf{x}^t\tilde{Q}\mathbf{x}$.  Throughout this
paper, the same notation will be used for a quadratic form and its
associated matrix.} It is denoted $SO^+(Q_1,\mathbb{R})$. This group
contains the discrete subgroup $SO^+(Q_1,\mathbb{Z})$ of integer
valued matrices.

Now consider a different quadratic form on $\mathbb{R}^4$; we will
be particularly interested in
\[Q_7(x_1,x_2,x_3,x_4)=x_1^2+x_2^2+x_3^2-7x_4^2.\]

Over $\mathbb{R}$, $Q_7$ and $Q_1$ are isomorphic; that is, there
exists an isometry $M:\mathbb{R}^4\rightarrow\mathbb{R}^4$ such that
$M^tQ_7M=Q_1$.  Therefore, $SO^+(Q_1,\mathbb{R})$ is isomorphic to
$SO^+(Q_7,\mathbb{R})$. However, there is no reason for there to be
an isomorphism between $SO^+(Q_7,\mathbb{Z})$ and
$SO^+(Q_1,\mathbb{Z})$.  What is true is that $SO^+(Q_7,\mathbb{Z})$
does act on $\mathbb{H}$ as a subgroup of
$SO^+(Q_7,\mathbb{R})=SO^+(Q_1,\mathbb{R})$.

All of the above is true for any integral quadratic form of
signature $(3,1)$.  Why, then, focus on the form $Q_7$?  The
following theorem sheds some light on this.

%
%
%

\begin{thm}\cite{Borel}.
$Q(\mathbf{x})=0$ has a non-zero integral solution if and only if
$\mathbb{H}^3/SO^+(Q,\mathbb{Z})$ is not compact.
\end{thm}

$Q_7$ does not have such a solution.  To see this, assume it does,
and furthermore, assume that the coordinates of $\mathbf{x}$ have no
common factor (no loss of generality because a common factor can
just be pulled out). Then consider the above equation mod 8.
\[0\equiv x_1^2+x_2^2+x_3^2-7x_4^2\equiv x_1^2+x_2^2+x_3^2+x_4^2
\mod 8\] However, mod 8, there are only three possible squares: 0, 1
and 4. In the above equation, none of the summands can be 1, because
then the rest would have to also be 1 to satisfy the equation mod 4,
which is not a solution mod 8.  Therefore, all the summands are
either 0 or 4, meaning $x_1$, $x_2$, $x_3$ and $x_4$ are all even,
contradicting the assumption of having no common factor.

Thus, the quotient orbifold $\mathbb{H}^3/SO^+(Q_7,\mathbb{Z})$ is
compact, and the goal now is to use this to understand the algebraic
structure of $SO^+(Q_7,\mathbb{Z})$.

It is important to note that the techniques used here to compute a
generating set for this group can be used to find similar generating
sets for $SO^+(Q_n,\mathbb{Z})$, with $Q_n=x_1^2+x_2^2+x_3^2-nx_4^2$
and $Q_n(\hat{x})=0$ having no non-zero solutions\footnote{The proof
that $Q_7$ was such a form also proves that $Q_n$ will also work,
for any $n=8m-1$, for $m>0$.}.  The basic algorithm will, in fact,
work for any integral quadratic form of signature $(3,1)$ which
yields a compact quotient orbifold, although the details will
require some work.

\section*{Dirichlet Fundamental Domains}

To further understand $\mathbb{H}^3/SO^+(Q_7,\mathbb{Z})$, the next
step will be to create a fundamental domain for it.  The fundamental
domain of choice will be a Dirichlet fundamental domain.


Constructing a Dirichlet fundamental domain $D$ for the action of a
group $G$ begins with choosing a point $x$ in hyperbolic space.
Then, determine two open sets, $U$ and $V$.

Let $U$ be a fundamental domain for the stabilizer of $x$.  In the
disc model, if the point $x$ is the central point, then the
stabilizer consists entirely of rotations, and so the fundamental
domain $U$ can be taken to be a cone.

Let $V$ be the set of points in hyperbolic space that are closer to
$x$ than any other point in $Gx$, the orbit of $x$.  This is often
called the 'Voronoi cell' of $x$. In practice, Voronoi cells can be
constructed in the following way. Let $A\in G$, and let $H_A$ denote
the half-space consisting of points that are closer to $x$ than to
$Ax$.  Then, the Voronoi cell is the intersection of $H_A$ over all
$A\in G$.

Then, the Dirichlet fundamental domain $D$ is $U\cap V$.

%

When the quotient orbifold is compact, a Dirichlet fundamental
domain can be computed algorithmically.

\textbf{Step 1)} Choose a fundamental domain $U$ for the stabilizer
of $x$.

\textbf{Step 2)} Produce $G_L$, the set of elements of $G$ that move
the point $x$ less than some distance $L$.

\textbf{Step 3)} Let $V_L$ be the intersection of $H_A$ over all
$A\in G_L$. If $V_L$ is compact, then there is some distance $M$
such that no point in $V_L$ is more than $M$ away from $x$.

\textbf{Step 4)} If $M<L/2$, then $V_L$ is the same as $V$, the
Voronoi cell of all of $G$, because $V_L\subset H_A$ for all $A$
that move $x$ more that $2M$. If $M\geq L/2$ or $V_L$ is not
compact, increase $L$ and do it again.

This will work eventually, because it must work when $L$ is greater
than the diameter of the orbifold.

How does this look for $\mathbb{H}^3/SO^+(Q_7,\mathbb{Z})$?  The
point $x$ will be $(0,0,0,7^{-1/2})$, and its stabilizer consists of
$SO^+(3,\mathbb{Z})$ acting on the first 3 coordinates.  It is a
group of order 24, and consists of compositions of permutations and
reflections along axes which are orientation preserving.
Fortunately, their action on the disc model with $x$ in the center
is exactly the same as their action on the first 3 coordinates of
the hyperboloid model, and so creating a fundamental domain is easy.
We will take the fundamental domain $U$ to be the set of points in
the disc which have positive Euclidean inner-product with the
following three vectors:
\[v_1=(1,-1,0,0);\;\;\;v_1=(0,1,-1,0);\;\;\;v_2=(0,0,1,0)\]


\subsection*{Generating Elements of $SO^+(Q_7,\mathbb{Z})$.}

%

The hard part about producing the rest of the fundamental domain is
producing all elements of $SO^+(Q_7,\mathbb{Z})$ which move $x$ less
than some fixed distance. To do this, several facts about elements
of $SO^+(Q_7,\mathbb{Z})$ will be helpful. Let $A$ be an element of
$SO^+(Q_7,\mathbb{Z})$, and let $a_{ij}$ denote the $ij$th
coefficient of the matrix $A$.

\begin{fact} The first three columns of the matrix $A$ have pseudolength
\footnote{The term 'pseudolength' of a vector $\mathbf{x}$ in
$\mathbb{R}^4$ just means $Q_7(\mathbf{x})$.  The term emphasizes
the idea of $Q_7$ giving $\mathbb{R}^4$ a pseudometric.} 1, the last
column has pseudolength -7, and they are all pairwise orthogonal.
Furthermore, any integral matrix that satisfies these properties is
automatically in $SO(Q_7,\mathbb{Z})$.
\end{fact}

\begin{fact} Let $x$ be an element of $\mathbb{H}^3$.
The distance in $\mathbb{H}^3$ from $x$ to $Ax$ is $\cosh(a_{44})$,
where $a_{44}$ is the bottom right entry in the matrix.
\end{fact}

\begin{fact} The entries $a_{14}$, $a_{24}$ and $a_{34}$, the right
column except for the bottom right entry, are all divisible by $7$.
\end{fact}

\begin{fact} The entry $a_{44}$ is either 1 more or 1 less than a
multiple of 7.
\end{fact}

\begin{fact} The absolute values of $a_{41}$, $a_{42}$ and $a_{43}$, the
bottom entries save the bottom right entry, are all bounded above by
$a_{44}/\sqrt{7}$.
\end{fact}

\begin{proof}
1) This is just a different way of saying $A^tQ_7A=Q_7$.

3) $A$ satisfies the equation $A^tQ_7A=Q_7$. $(Q_7)^{-1}A^tQ_7A=1$,
and so $A^{-1}=(Q_7)^{-1}A^tQ_7$.  Given the simple nature of $Q_7$
written as a symmetric matrix, this has a simple form:
\[A^{-1}=\left(
           \begin{array}{cccc}
             a_{11} & a_{21} & a_{31} & -7a_{41} \\
             a_{12} & a_{22} & a_{32} & -7a_{42} \\
             a_{13} & a_{23} & a_{33} & -7a_{43} \\
             -a_{14}/7 & -a_{24}/7 & -a_{34}/7 & a_{44} \\
           \end{array}
\right)\] Since this matrix also consists of integers, the specified
entries are divisible by 7.

4) The right column has pseudolength $-7$, which means
$a_{14}^2+a_{24}^2+a_{34}^2-7a_{44}^2=-7$. Using their divisiblity
by 7, let $a_{14}=7\alpha$, $a_{24}=7\beta$ and $a_{34}=7\gamma$.
Then $49(\alpha^2+\beta^2+\gamma^2)-7a_{44}^2=-7$, and so
$7(\alpha^2+\beta^2+\gamma^2)=a_{44}^2-1=(a_{44}-1)(a_{44}+1)$.  $7$
must divide $a_{44}-1$ or $a_{44}+1$, proving the statement.

5) Since $A$ is the inverse of some other matrix in
$SO^+(q,\mathbb{Z})$, its entries satisfy the equation
$7(a_{41}^2+a_{42}^2+a_{43}^2)=a_{44}^2-1$.  Therefore,
$\max(|a_{41}|,|a_{42}|,|a_{43}|)\leq a_{44}/\sqrt{7}$.
\end{proof}

\subsection*{The Algorithm for Generating Such Elements.}

With these facts established, elements of $SO^+(Q_7,\mathbb{Z})$
which move $x$ less than some fixed distance $\cosh(L)$ can be
generated.

\textbf{Step 1)} Generate a list of pseudolength 1 vectors whose
fourth component is bounded above by $L/\sqrt{7}$.  This list is
finite since the sum of the squares of the first three coordinates
is bounded above by $L^2+1$.

\textbf{Step 2)} For every $0\leq n:=(7i\pm1)\leq L$ (for $i$ an
integer) find all the solutions of $(7a)^2+(7b)^2+(7c)^2-7n^2=-7$,
and correspondingly, all the possible right columns of elements of
$SO^+(Q_7,\mathbb{Z})$.  Again, there are only a finite number of
possibilities.

\textbf{Step 3)} For each possible right column found in Step 2,
compare it to every vector on the list from Step 1, to find the
vectors which are orthogonal to the chosen right column.  Of the
vectors which are orthogonal, compare them to each other and find
which pairs are orthogonal to each other.  Then, any triple of
vectors which are orthogonal to the chosen right column and each
other constitutes an element of $SO^+(Q_7,\mathbb{Z})$.

Once the elements have been generated, the half planes they generate
must be drawn, and compactness and maximum distance must be checked.


\section*{The Generators}

So what does this algorithm yield?  First, in the interest of coming
up with a generating set for $SO^+(Q_7,\mathbb{Z})$, we need a
generating set of the stabilizer of the distinguished point.  As was
mentioned before, this group is $SO^+(3,\mathbb{Z})$, which is known
to be isomorphic to $S_4$.

To see this, notice that this group is the symmetry group of the
cube, and that a symmetry of the cube is determined by its
permutation of the diagonals of the cube.  There are four diagonals
in a cube, and every permutation of the diagonals is possible.
Generators for this group as a subgroup of $SO^+(Q_7,\mathbb{Z})$
are:

\[(12)=\left(
  \begin{array}{cccc}
    0 & 1 & 0 & 0 \\
    1 & 0 & 0 & 0 \\
    0 & 0 & -1 & 0 \\
    0 & 0 & 0 & 1 \\
  \end{array}
\right),\;\;\;(1234)=\left(
  \begin{array}{cccc}
    0 & 0 & 1 & 0 \\
    0 & 1 & 0 & 0 \\
    -1 & 0 & 0 & 0 \\
    0 & 0 & 0 & 1 \\
  \end{array}
\right)\]

The above algorithm yields the following as a sufficient set of
generators.  There is some ambiguity in which matrices to choose, up
to multiplication by an element of the stabilizer.  These
particularly matrices have been chosen so that they move the
Dirichlet fundamental domain to an adjacent region.

\[A=\left(
  \begin{array}{cccc}
    -2 & -2 & 0 & 7 \\
    -5 & -2 & 0 & 14 \\
    0 & 0 & -1 & 0 \\
    -2 & -1 & 0 & 6 \\
  \end{array}
\right) A^{-1}=\left(
  \begin{array}{cccc}
    -2 & -5 & 0 & 14 \\
    -2 & -2 & 0 & 7 \\
    0 & 0 & -1 & 0 \\
    -1 & -2 & 0 & 6 \\
  \end{array}
\right)\]

\[B=\left(
  \begin{array}{cccc}
    0 & -1 & 0 & 0 \\
    -8 & 0 & 0 & 21 \\
    0 & 0 & -1 & 0 \\
    -3 & 0 & 0 & 8 \\
  \end{array}
\right) C=\left(
  \begin{array}{cccc}
    -4 & -3 & -2 & 14 \\
    -3 & -4 & -2 & 14 \\
    -2 & -2 & 0 & 7 \\
    -2 & -2 & -1 & 8 \\
  \end{array}
\right) B^{-1}=\left(
  \begin{array}{cccc}
    0 & -8 & 0 & 21 \\
    -1 & 0 & 0 & 0 \\
    0 & 0 & -1 & 0 \\
    0 & -3 & 0 & 8 \\
  \end{array}
\right)\]

\[D=\left(
  \begin{array}{cccc}
    -4 & -3 & -2 & 14 \\
    -9 & -4 & -4 & 28 \\
    -4 & -2 & -3 & 14 \\
    -4 & -2 & -2 & 13 \\
  \end{array}
\right) D^{-1}=\left(
  \begin{array}{cccc}
    -4 & -9 & -4 & 28 \\
    -3 & -4 & -2 & 14 \\
    -2 & -4 & -3 & 14 \\
    -2 & -4 & -2 & 13 \\
  \end{array}
\right)\]

\[E=\left(
  \begin{array}{cccc}
    -7 & -8 & 0 & 28 \\
    -8 & -7 & 0 & 28 \\
    0 & 0 & -1 & 0 \\
    -4 & -4 & 0 & 15 \\
  \end{array}
\right)\]

\[F=\left(
  \begin{array}{cccc}
    -10 & -12 & -10 & 49 \\
    -2 & -4 & -3 & 14 \\
    -3 & -4 & -2 & 14 \\
    -4 & -5 & -4 & 20 \\
  \end{array}
\right) F^{-1}=\left(
  \begin{array}{cccc}
    -10 & -2 & -3 & 28 \\
    -12 & -4 & -4 & 35 \\
    -10 & -3 & -2 & 28 \\
    -7 & -2 & -2 & 20 \\
  \end{array}
\right)\]

Notice that a left-right reflection of this list sends matrices to
their inverses.

Which of these matrices are necessary to generate
$SO^+(Q_7,\mathbb{Z})$?  The following relations demonstrate that
only $A$, $C$ and the stabilizer group are necessary.
\[B=(1234)^2(12)A^{-2}(12)\]
\[D=(1234)^3(12)C(1234)^3\]
\[E=A^{-1}(1234)^2(12)A^{-1}(12)\]
\[F=C(12)(1234)^2A^{-1}(12)\]

Therefore, the four matrices $(12)$, $(1234)$, $A$ and $C$ are
enough to generate the entire group $SO^+(Q_7,\mathbb{Z})$.

\end{document}